\documentclass[12pt]{amsart}
\usepackage{amssymb}
\usepackage{graphicx}
\usepackage{t1enc}
\usepackage[latin2]{inputenc}
\usepackage{verbatim}
\usepackage{amsmath,amsfonts,amssymb,amsthm}
\usepackage[mathcal]{eucal}
\usepackage{enumerate}

\newtheorem{theorem}{Theorem}[section]

\newtheorem{lemma}[theorem]{Lemma}

\numberwithin{equation}{section}

\begin{document}
\author[K. Nagy]{K\'aroly Nagy}
\title[The existence and unicity]{The existence and unicity of numerical solution of initial value problems by Walsh polynomials approach}
\thanks{Research supported by project GINOP-2.2.1-15-2017-00055.}

\address{K. Nagy, Institute of Mathematics and Computer Sciences, University of Ny\'\i
regyh\'aza, P.O. Box 166, Ny\'\i regyh\'aza, H-4400 Hungary }
\email{nagy.karoly@nye.hu}

\date{}

\begin{abstract} 
Chen and  Hsiao gave the numerical solution of initial value
problems of systems of linear differential equations with constant coefficients by Walsh polynomials approach. This result was improved by G\'at and Toledo for initial value problems of differential equations with   variable coefficients on the interval
$[0,1[$ and initial value $\xi=0$. In the present paper we discuss the general case while $\xi$ can take any arbitrary value in the interval $[0,1[$. We show the existence and uniform convergence of the numerical solution, as well. 
\end{abstract}
\maketitle

\section{Introduction}
In 1973  Corrington developed
a method to solve $n$th order linear differential equations \cite{Corr}, he used huge tables of the Walsh-Fourier coefficients of certain 
integrals of Walsh functions.
Two years later   Chen and  Hsiao created a new
procedure for the numerical solution of initial value
problems of systems of linear differential equations with constant coefficients by Walsh polynomials approach  \cite{CH1}, they improved the method of Corrington. 
In this period Chen and Hsiao wrote several papers in which they showed the applicability 
of their procedure in different fields of sciences \cite{CH2, CH3, CH4}. Applying this method several papers were born \cite{CH-V, SH, O, R}.
However, the authors 
did not deal with the analysis of the proposed numerical solution. 

Recently, the method of Chen and Hsiao was analysed by G\'at and Toledo \cite{Gat-Toledo, GT2} using the tools of the theory of dyadic harmonic analysis \cite{SchippBook}. They investigated the solvability of the linear system appearing during the procedure of Chen and Hsiao and the  estimation  of errors.  
In  paper \cite{GT} the authors extended the results in \cite{Gat-Toledo} to develop a similar method for
solving initial value problems of differential equations with not necessarily constant coefficients. The existence and unicity of the numerical solution were discussed. Moreover,  estimation of errors was given. Namely, the Cauchy problem 
\[
\begin{split}
	y'+p(x)y &=q(x),\\
	y(0)&=\eta,
\end{split}
\]
was treated with some assumption on functions $p(x),q(x)$. Some useful computations were improved in \cite{SM}. It is important to note that 
not only the Walsh polynomials are applied for numerical solution of differential equations with initial value condition. Several papers were written for other orthonormal systems, mainly for Haar system e.g. \cite{LLT,L}, as well. The biggest difference between the Walsh and Haar system is that, while the Walsh system is bounded and takes only two values +1 and -1, the Haar system is unbounded and could take very big values, as well.  Handling the Walsh system is more effective and its application eliminates the errors of calculations in sense of programming. 

In some problem it is impossible to establish an initial value at point $x=0$. For example  we consider the differential equation 
$$
y'+x^2 y=1-\frac{2}{x^3}.
$$
The general solution is $y=\frac{1}{x^2}+Ce^{-\frac{x^3}{3}}$. But $y(x)$ and $ q(x)$ is not determined at point $x=0$. So, we can not start the solution from such a type point, where the abscissa is 0. So, it seems to be natural to choose another starting point. 
For example, we could discuss the Cauchy problem
\begin{equation}\label{example}
\begin{aligned}
y'+x^2 y&=1-\frac{2}{x^3}\\
y\left(\frac{1}{2}\right)&=4.
\end{aligned}
\end{equation}
Its exact solution is $y(x)=\frac{1}{x^2}$.  
We could choose initial value  in a general form 
$
y(\xi)=\eta$, where $\xi\in [0,1[$.
Motivating by the previous problem \eqref{example} we deal with the Cauchy problem 
\begin{equation}\label{Cauchy}
\begin{split}
y'+p(x)y &=q(x),\\
y(\xi)&=\eta,
\end{split}
\end{equation}
where $p,q\colon [0,1[ \to {\mathbb R}$ are continuous functions with
$$
\int_0^1 |p(x)|dx<\infty, \quad \int_0^1 |q(x)| dx<\infty
$$
and $\xi\in [0,1[$.

The equivalent integral equation is
\begin{equation}\label{integral-eq}
{y}(x) = \eta +\int_{\xi}^{x} q(t)-p(t){y}(t) dt ,\quad x\in [0,1[.
\end{equation}
The connected discretized integral equation is given in the form 
\begin{equation}\label{discrete}
\overline{y}_n(x) = \eta +S_{2^n}\left(\int_{\xi}^{.} S_{2^n}q(t)-S_{2^n}p(t)\overline{y}_n(t) dt\right) (x),\quad x\in [0,1[,
\end{equation}
where $\overline{y}_n$ denotes a Walsh polynome of the form $\overline{y}_n=\sum_{k=0}^{2^n-1}c_k\omega_k$. Our aims are to determine the Walsh polynome $\overline{y}_n$ by a very fast numerical algorithm (so called multistep method) and after this to show the unicity of this solution. Moreover, we estimate the error of the numerical solution. At last, we present an example to illustrate the effectiveness of our multistep method.    In our main theorem we investigate the uniform convergence of the numerical solution on the interval $[0,1[ $.

\section{Definitions and notations}

Every $n\in\mathbf{N}$ can be uniquely expressed in the number system based 2 by 
\[
n=\sum_{k=0}^{\infty}n_k2^k,
\]
where $n_k\in \{ 0,1\}$ for all $k\in\mathbb{N}$. The sequence $(n_0,n_1,\ldots )$  called the dyadic expansion of $n$. Analogously, the dyadic expansion $(x_0,x_1,\dots)$ of a real number $x\in[0,1[$ is determined by the sum
\begin{equation*}
x=\sum_{k=0}^{\infty}\frac{x_k}{2^{k+1}},
\end{equation*}
where $x_k\in\{ 0,1\}$ for all $k\in\mathbb{N}$. This expansion is not unique if $x$ is a dyadic rational, i.e. $x$ is a number of the form $\frac{i}{2^{k}}$, where $i,k\in\mathbb N$ and $0\le i<2^k$. For dyadic rationals we choose the expansion terminates in zeros. Define the dyadic sum of two numbers $x,y\in[0,1[$ with expansion $(x_0,x_1,\dots)$ and $(y_0,y_1,\ldots)$, respectively by
\begin{equation*}
x\dotplus y:=\sum_{k=0}^{\infty}|x_k-y_k|2^{-(k+1)}.
\end{equation*}
The Rademacher functions are defined by 
\begin{equation*}
r_k(x):=(-1)^{x_k}\quad(x\in[0,1[,\ k\in\mathbb N).
\end{equation*}
The Walsh system in the Paley enumeration is defined as the product system of Rademacher functions
\begin{equation*}
\omega_n(x):=\prod_{k=0}^{\infty}r_k^{n_k}(x)\qquad(x\in[0,1[, n\in\mathbb N).
\end{equation*}
It is known that the Walsh-Paley system is complete orthonormal system in $L^2([0,1[)$ \cite{SchippBook}.
For an integrable function $f\in L^1([0,1[)$, 
the Fourier coefficients and partial sums of Fourier series are defined by
\begin{align*}
\widehat f_k&:=\int_{0}^{1}f(x)\omega_k(x)\,dx\quad(k\in\mathbb N),\\
S_nf(x)&:=\sum_{k=0}^{n-1}\widehat f_k\omega_k(x)\quad(n\in\mathbb N, x\in[0,1[).
\end{align*}

The  $n$th Dirichlet kernel is defined by 
\[
D_{n}(x):=\sum_{k=0}^{n-1}\omega_k(x)\quad(x\in[0,1[)
\]
The $2^n$th Dirichlet kernel has the following well known property (see \cite{SchippBook})
\begin{equation}\label{EqPaleyLemma}
D_{2^n}(x)=
\begin{cases}
2^n, &0\le x<\frac{1}{2^n},\\
0, &\frac{1}{2^n}\le x<1.
\end{cases}
\end{equation}
This yields that 
the $2^{n}$-th partial sums can be written in the form 
\[
S_{2^n}f(x)=2^n\int_{I_n(x)}f(y)\,dy
\]
where the sets 
\begin{equation*}\label{EqDyadicIntervals}
I_{n}(i):=\left[\frac{i-1}{2^{n}},\frac{i}{2^{n}}\right[\quad (i=1,\dots,2^{n})
\end{equation*}
are called dyadic intervals,  and $I_n(x)$ denotes the dyadic interval which contains  $x$ ($x\in [0,1[$).

It is important to note that 
 $S_{2^n}f$ converges to $f$ in $L^1$-norm for every integrable function $f$ (see \cite{SchippBook} p. 142).

The matrix $A$ of size $2^{n}$ is called the dyadic circulant matrix generated by the numbers $a_0,a_1,\ldots,a_{2^{n}-1}$ if for all of the entries of the matrix $A$
\[
a_{i,j}=a_{i\oplus j}\quad(i,j=0,1,\ldots,2^{n}-1)
\]
holds, where $a_{i,j}$ is  in the $i$-th row and $j$-th column of $A$, and $i\oplus j$ denotes the dyadic sum of the non-negative integers $i$ and $j$.
Let us define the function
\[
a(x)=\sum_{j=0}^{2^{n}-1}a_j\omega_j(x)\quad(x\in[0,1[).
\]
In paper \cite[Lemma 2]{GT} it is proved that the dyadic circulant matrix $A$  can be written as
\begin{equation}\label{EqDiagonalization}
A=WD_{a}W^{-1},
\end{equation}
where the matrix
\[
D_a=
\begin{pmatrix}
a(0)&0&0&\dots&0\\
0&a(\frac{1}{2^n})&0&\dots&0\\
0&0&a(\frac{2}{2^n})&\dots&0\\
\vdots&\vdots&\vdots&\ddots&\vdots\\
0&0&0&\dots&a(\frac{2^n-1}{2^n})
\end{pmatrix}
\]
is diagonal and the matrix
\[
W=
\begin{pmatrix}
\omega_0(0)&\omega_1(0)&\omega_1(0)&\dots&\omega_{2^n-1}(0)\\[2pt]
\omega_0(\frac{1}{2^n})&\omega_1(\frac{1}{2^n})&\omega_1(\frac{1}{2^n})&\dots&\omega_{2^n-1}(\frac{1}{2^n})\\[2pt]
\omega_0(\frac{2}{2^n})&\omega_1(\frac{2}{2^n})&\omega_1(\frac{2}{2^n})&\dots&\omega_{2^n-1}(\frac{2}{2^n})\\
\vdots&\vdots&\vdots&&\vdots\\
\omega_0(\frac{2^n-1}{2^n})&\omega_1(\frac{2^n-1}{2^n})&\omega_1(\frac{2^n-1}{2^n})&\dots&\omega_{2^n-1}(\frac{2^n-1}{2^n})\\
\end{pmatrix}
\]
is the Hadamard matrix of size $2^n\times 2^n$ derived from the Walsh-Paley system (see \cite{SchippBook}, as well).
It is natural to say that the dyadic circulant matrix $A$ is generated by the Walsh polynome $a(x)$.

Triangular functions $J_k^\xi$ are the integral function of the Walsh-Paley functions $\omega_k$. That is, 
\[
J_k^\xi(x):=\int_{\xi}^{x}\omega_k(t)\,dt\quad(k\in\mathbb N,0\le x<1).
\]
Let  $\widehat{J^\xi}_{k,j}$ be the $j$th Walsh-Fourier coefficient of the triangular function $J_k^\xi$. We can find the exact calculation of the values of $\widehat{J^0}_{k,j}$
in \cite{Gat-Toledo} directly by the Fine's formulae (see \cite{Fine}).
Let $\widehat{J^\xi}^{(n)}$ be the matrices whose entries are $\widehat{J^\xi}_{k,j}$, where $k,j=0,1,\dots,2^n-1$. Simply we write $\widehat{J^\xi}$.
We note that  
$$
J_k^\xi(x)=\int_{0}^{x}\omega_k(t)dt-\int_{0}^{\xi}\omega_k(t)dt
=J_k^0(x)-J_k^0(\xi)
$$
for all $0\leq \xi<x<1$. 

At last we note that, in this paper we follow the notation of paper  G\'at and Toledo \cite{GT}.

\section{Multistep algorithm based on the integral equation}
In this section, we consider the Walsh polynomials
\begin{equation}\label{polinom}
\overline{y}_n(x)=\sum_{k=0}^{2^n-1}c_k\omega_k(x)
\end{equation}
satisfying the discretized integral equation \eqref{discrete}.

In order to simplify our notations we denote by 
$\tilde{q}_n:=S_{2^n}q$ and $\tilde{p}_n:=S_{2^n}p$.
Since, the functions $\overline{y}_n(x),\ S_{2^n}q(x),\ S_{2^n}p(x)$ 
are constant on the dyadic intervals $I_n(i)=[\frac{i-1}{2^n},\frac{i}{2^n}[$ for all $i=1,2,\ldots, 2^n$, we write 
\begin{equation*}
\tilde{q}_n(x)-\tilde{p}_n(x)\overline{y}_n(x)=
\sum_{k=1}^{2^n} \left( \tilde{q}_n(\frac{k-1}{2^n})-\tilde{p}_n(\frac{k-1}{2^n})\overline{y}_n(\frac{k-1}{2^n})\right)\chi_{I_n(k)}(x).
\end{equation*}
Then  the discretized integral equation \eqref{discrete} could be written in the form
\begin{eqnarray}
\overline{y}_n(x) &=& \eta +S_{2^n}\left(\int_{\xi}^{.}
\sum_{k=1}^{2^n} \left( \tilde{q}_n(\frac{k-1}{2^n})-\tilde{p}_n(\frac{k-1}{2^n})\overline{y}_n(\frac{k-1}{2^n})\right)\chi_{I_n(k)}(t)dt
\right) (x)\nonumber\\
&=&\eta +\sum_{k=1}^{2^n} \left( \tilde{q}_n(\frac{k-1}{2^n})-\tilde{p}_n(\frac{k-1}{2^n})\overline{y}_n(\frac{k-1}{2^n})\right)
S_{2^n}\left(\int_{\xi}^{.}
\chi_{I_n(k)}(t)dt \right) (x).\label{constant-1}
\end{eqnarray}

Now, we calculate the functions $f(s):=\int_{\xi }^{s} \chi_{I_n(k)}(t)dt$, $S_{2^n}\left(\int_{\xi}^{.}
\chi_{I_n(k)}(t)dt \right) (x)$.
We have three cases with respect to the value of $\xi$.

First, we set $0\leq\xi<(k-1)/2^n$. Then we get
\begin{equation}\label{alatt}
f_n(s)=
\begin{cases}
0, & 0\leq s<\frac{k-1}{2^n},\\
s-\frac{k-1}{2^n}, & \frac{k-1}{2^n}\leq s< \frac{k}{2^n},\\
\frac{1}{2^n},& \frac{k}{2^n}\leq s<1,
\end{cases}
\quad 
S_{2^n}(f_n)(x)=\begin{cases}
0,& 0\leq x<\frac{k-1}{2^n}, \\
\frac{1}{2^{n+1}}, &  \frac{k-1}{2^n}\leq x< \frac{k}{2^n},\\
\frac{1}{2^n},& \frac{k}{2^n}\leq x<1.
\end{cases}
\end{equation}
Second, we set $(k-1)/2^n\leq \xi <k/2^n$. Then 
\begin{equation}\label{között}
f_n(s)=
\begin{cases}
\frac{k-1}{2^n}-\xi, & 0\leq s<\frac{k-1}{2^n},\\
s-\xi, & \frac{k-1}{2^n}\leq s< \frac{k}{2^n},\\
\frac{k}{2^n}-\xi,& \frac{k}{2^n}\leq s<1,
\end{cases}
\quad 
S_{2^n}(f_n)(x)=\begin{cases}
\frac{k-1}{2^n}-\xi,& 0\leq x<\frac{k-1}{2^n}, \\
\frac{2k-1}{2^{n+1}}-\xi, &  \frac{k-1}{2^n}\leq x< \frac{k}{2^n},\\
\frac{k}{2^n}-\xi,& \frac{k}{2^n}\leq x<1.
\end{cases}
\end{equation}
At last, we set $k/2^n \leq \xi $. We have
\begin{equation}\label{felett}
f_n(s)=
\begin{cases}
-\frac{1}{2^n} & 0\leq s<\frac{k-1}{2^n},\\
s-\frac{k}{2^n}, & \frac{k-1}{2^n}\leq s< \frac{k}{2^n},\\
0,& \frac{k}{2^n}\leq s<1,
\end{cases}
\quad 
S_{2^n}(f_n)(x)=\begin{cases}
-\frac{1}{2^n},& 0\leq x<\frac{k-1}{2^n}, \\
-\frac{1}{2^{n+1}}, &  \frac{k-1}{2^n}\leq x< \frac{k}{2^n},\\
0,& \frac{k}{2^n}\leq x<1.
\end{cases}
\end{equation}
There exists $k^* \in \{1,\ldots, 2^n\}$, such that $\xi \in I_n(k^*)$ ($k^*$  depends on $n$, that is $k^*=k^*(n)$). We divide the sum in equation \eqref{constant-1} into three parts as follows 
\begin{eqnarray}
\overline{y}_n(x) &=&\eta +\sum_{k=1}^{k^*-1} \left( \tilde{q}_n(\frac{k-1}{2^n})-\tilde{p}_n(\frac{k-1}{2^n})\overline{y}_n(\frac{k-1}{2^n})\right)
S_{2^n}\left(\int_{\xi}^{.}
\chi_{I_n(k)}(t)dt \right) (x)\nonumber\\
&&+ \left( \tilde{q}_n(\frac{k^*-1}{2^n})-\tilde{p}_n(\frac{k^*-1}{2^n})\overline{y}_n(\frac{k^*-1}{2^n})\right)
S_{2^n}\left(\int_{\xi}^{.}
\chi_{I_n(k^*)}(t)dt \right) (x)\label{constant-2}\\
&&+\sum_{k=k^*+1}^{2^n} \left( \tilde{q}_n(\frac{k-1}{2^n})-\tilde{p}_n(\frac{k-1}{2^n})\overline{y}_n(\frac{k-1}{2^n})\right)
S_{2^n}\left(\int_{\xi}^{.}
\chi_{I_n(k)}(t)dt \right) (x)\nonumber
\end{eqnarray}
Now, we set $x\in I_n(i)$. We have three cases determined by the relation between $i,k^*$.

\emph{Case I.} $k^*<i$ (that is, $\xi<x$ and $\xi,x$ lay in different dyadic intervals).
Since $\tilde{y}_n$ is constant on the interval $I_n(i)$, we may write 
$\tilde{y}_n(x)=\tilde{y}_n(\frac{i-1}{2^n})$. From equality \eqref{alatt}-\eqref{constant-2}, we immediately write 
\begin{eqnarray*}
\tilde{y}_n(\frac{i-1}{2^n})&=&\eta +
\sum_{k=1}^{k^*-1}\left( \ldots\right) 0
+\left( \tilde{q}_n(\frac{k^*-1}{2^n})-\tilde{p}_n(\frac{k^*-1}{2^n})\overline{y}_n(\frac{k^*-1}{2^n})\right)(\frac{k^*}{2^n}-\xi)\\
&&+\sum_{k=k^*+1}^{i-1} \left( \tilde{q}_n(\frac{k-1}{2^n})-\tilde{p}_n(\frac{k-1}{2^n})\overline{y}_n(\frac{k-1}{2^n})\right)\frac{1}{2^n}\\
&&+\left( \tilde{q}_n(\frac{i-1}{2^n})-\tilde{p}_n(\frac{i-1}{2^n})\overline{y}_n(\frac{i-1}{2^n})\right)\frac{1}{2^{n+1}}.
\end{eqnarray*}
This yields 
\begin{eqnarray*}
\tilde{y}_n(\frac{i-1}{2^n})=
\frac{1}{1+\frac{\tilde{p}_n(\frac{i-1}{2^n})}{2^{n+1}}}\left( 
\eta+\left( \tilde{q}_n(\frac{k^*-1}{2^n})-\tilde{p}_n(\frac{k^*-1}{2^n})\overline{y}_n(\frac{k^*-1}{2^n})\right)(\frac{k^*}{2^n}-\xi)\right.\\
\left. +\sum_{k=k^*+1}^{i-1} \left( \tilde{q}_n(\frac{k-1}{2^n})-\tilde{p}_n(\frac{k-1}{2^n})\overline{y}_n(\frac{k-1}{2^n})\right)\frac{1}{2^n}
+\frac{1}{2^{n+1}}\tilde{q}_n(\frac{i-1}{2^n})
\right).
\end{eqnarray*}
Thus, it is easy to obtain a recursive algorithm starting from the value $\overline{y}_n(\frac{k^*-1}{2^n})$, if it is known. See \emph{Case III}.

\emph{Case II.} $k^*>i$ (that is, $x<\xi$ and $\xi,x$ lay in different dyadic intervals).
Equality \eqref{alatt}-\eqref{constant-2} yield
\begin{eqnarray*}
\tilde{y}_n(\frac{i-1}{2^n})&=&\eta +
\sum_{k=1}^{i-1}\left( \ldots\right) 0
+\left( \tilde{q}_n(\frac{i-1}{2^n})-\tilde{p}_n(\frac{i-1}{2^n})\overline{y}_n(\frac{i-1}{2^n})\right)\frac{-1}{2^{n+1}}\\
&&+\sum_{k=i+1}^{k^*-1} \left( \tilde{q}_n(\frac{k-1}{2^n})-\tilde{p}_n(\frac{k-1}{2^n})\overline{y}_n(\frac{k-1}{2^n})\right)\frac{-1}{2^n}\\
&&+\left( \tilde{q}_n(\frac{k^*-1}{2^n})-\tilde{p}_n(\frac{k^*-1}{2^n})\overline{y}_n(\frac{k^*-1}{2^n})\right)(\frac{k^*-1}{2^{n}}-\xi)\\
&&+
\sum_{k=k^*+1}^n\left( \ldots \right) 0.
\end{eqnarray*}
By this we could express the value  $\tilde{y}_n(\frac{i-1}{2^n})$ in the form 
\begin{eqnarray*}
\tilde{y}_n(\frac{i-1}{2^n})=
\frac{1}{1-\frac{\tilde{p}_n(\frac{i-1}{2^n})}{2^{n+1}}}\left( 
\eta+\left( \tilde{q}_n(\frac{k^*-1}{2^n})-\tilde{p}_n(\frac{k^*-1}{2^n})\overline{y}_n(\frac{k^*-1}{2^n})\right)(\frac{k^*-1}{2^n}-\xi)\right.\\
\left. -\sum_{k=i+1}^{k^*-1} \left( \tilde{q}_n(\frac{k-1}{2^n})-\tilde{p}_n(\frac{k-1}{2^n})\overline{y}_n(\frac{k-1}{2^n})\right)\frac{1}{2^n}
-\frac{1}{2^{n+1}}\tilde{q}_n(\frac{i-1}{2^n})
\right).
\end{eqnarray*}
We obtain a recursive algorithm starting from the value $\overline{y}_n(\frac{k^*-1}{2^n})$ and $i$ goes down to 1, if $\overline{y}_n(\frac{k^*-1}{2^n})$ is known. See \emph{Case III}.

\emph{Case III.} $k^*=i$ (that is, $\xi,x$ belong to the same dyadic interval). We apply 
equality \eqref{alatt}-\eqref{constant-2}, again. 
\begin{eqnarray*}
	\tilde{y}_n(\frac{k^*-1}{2^n})&=&\eta +
\left( \tilde{q}_n(\frac{k^*-1}{2^n})-\tilde{p}_n(\frac{k^*-1}{2^n})\overline{y}_n(\frac{k^*-1}{2^n})\right)(\frac{2k^*-1}{2^{n+1}}-\xi).
\end{eqnarray*}
From this we could express the required value of $\tilde{y}_n(\frac{k^*-1}{2^n})$. That is,
$$
\tilde{y}_n(\frac{k^*-1}{2^n})=\frac{1}{1+\tilde{p}_n(\frac{k^*-1}{2^n})(\frac{2k^*-1}{2^{n+1}}-\xi) }
\left( \eta+ \tilde{q}_n(\frac{k^*-1}{2^n})
(\frac{2k^*-1}{2^{n+1}}-\xi)
\right).
$$
At last, we could state the following theorem.
\begin{theorem}\label{theorem-1}Let $p$ and $q$ be two integrable and continuous functions defined on the interval $[0,1[$. Then
there exists a $n^*\in {\mathbb N}$, such that the  discretized integral equations \eqref{discrete} for all $n\geq n^*$  with assumption of the original initial value problem \eqref{Cauchy} has got at least one solution.
\end{theorem}
\begin{proof}
The proof is based on the multistep algorithm presented above. 
First, we have to find a natural number $n^*$, such that the expressions $\tilde{y}_{n^*}(\frac{i-1}{2^{n^*}})$ could be calculated for all $i=1,\ldots, 2^{n^*}$. 
In paper \cite{GT} the next result is proved under the assumption that $p$ is continuous  and integrable on the interval $[0,1[$. 
\begin{equation}\label{main-eq}
\lim_{n\to \infty }\max_{0\leq i<2^{n}}\left\{ \frac{1}{2^{n}}\left|\tilde{p}_{n}(\frac{i}{2^{n}})\right| \right\}
=0.
\end{equation}
Applying this statement we could choose a natural number $n^*$, such that 
\begin{equation}\label{main-eq-2}
\max_{0\leq i<2^{n}}\left\{ \frac{1}{2^{n}}\left|\tilde{p}_{n}(\frac{i}{2^{n}})\right| \right\} <\frac{1}{2}
\end{equation}
holds for all $n\geq n^*$. Let us set $n\geq n^*$.
That is, 
\begin{equation}\label{main-eq-3}
1\pm \frac{\tilde{p}_n(\frac{i-1}{2^n})}{2^{n+1}}>3/4
\end{equation}
for all $i=1,\ldots, 2^n$. 
Now, we give $k^*$ in that way $\xi\in I_{n}(k^*)$ (that is $k^*=k^*(n)$).
Since, $\frac{2k^*-1}{2^{n+1}}$ is the middle point of the dyadic interval $I_{n}(k^*)$, we
have
$$
1+\tilde{p}_{n}(\frac{k^*-1}{2^{n}})(\frac{2k^*-1}{2^{{n}+1}}-\xi) >3/4
.$$
That is, $\tilde{y}_{n}(\frac{i-1}{2^{n}})$ are well defined for all $i=1,\ldots, 2^{n}$.

First, we determine $\tilde{y}_{n}(\frac{k^*-1}{2^{n}})$, by the formula given in \emph{Case III}.
After this, by the recursive formula in \emph{Case I} we start from $i=k^*$ up to $i=2^{n}$. 
At last, by the recursive formula in \emph{Case II} we calculate from $i=k^*$ down to $i=1$. 
\end{proof}
\section{Unicity of  solution of  discretized integral equation \eqref{discrete} }

In the previous section, we considered the Walsh polynomials 
$
\overline{y}_n(x)=\sum_{k=0}^{2^n-1}c_k\omega_k(x)
$
satisfying the discretized integral equation.
In this section our aim is to find the coefficients of the Walsh polynomial $\overline{y}_n$ for a fixed natural number $n\geq n^*$ ($n^*$ is determined in Theorem \ref{theorem-1}) and we show the unicity of this solution. 
We introduce the following vectors and matrices:
\begin{align*}
\mathbf{c}^\top&:=(c_0,c_1,\dots,c_{2^n-1}),\\
\mathbf{\widehat{q}}^\top&:=(\widehat{q}_0,\widehat{q}_1,\dots,\widehat{q}_{2^n-1}),\\
\mathbf{\widehat{p}}^\top&:=(\widehat{p}_0,\widehat{p}_1,\dots,\widehat{p}_{2^n-1}),\\
 \boldsymbol\omega(x)^\top&:=(\omega_0(x),\omega_1(x),\dots,\omega_{2^n-1}(x)),\\ \mathbf{e_0}^\top&:=(1,0,\dots,0)\textrm { with size }2^n.
\end{align*}
\begin{align*}
\widehat{J^\xi}&:=(\widehat{J^\xi}_{k,j})_{k,j=0}^{2^n-1},\quad
P:=(\widehat{p}_{i\oplus j})_{i,j=0}^{2^n-1},
\end{align*}
where $P$ is the dyadic circulant matrix generated by $S_{2^n}p(x)$.

The discretized integral equation $\eqref{discrete}$ can be written by the help of matrix notations as follows
\begin{equation}\label{EqMatrix}
\boldsymbol\omega(x)^\top \mathbf{c}
=\eta+S_{2^n}\left(\int_{\xi}^{.}\boldsymbol\omega(t)^\top\mathbf{\widehat{q}}-\boldsymbol\omega(t)^\top\mathbf{\widehat{p}}\boldsymbol\omega(t)^\top \mathbf{c}\,dt\right)(x)
\end{equation}
In paper \cite{GT} it is proved that 
$$
\boldsymbol\omega(t)^\top\mathbf{\widehat{p}}\boldsymbol\omega(t)^\top \mathbf{c}
=\boldsymbol\omega(t)^\top P \mathbf{c}.
$$
Using this we write equation \eqref{EqMatrix} in the next form

\begin{equation}\label{EqMatrix2}
\begin{aligned}
\boldsymbol\omega(x)^\top \mathbf{c}
&=\eta+S_{2^n}\left(\int_{\xi}^{.}\boldsymbol\omega(t)^\top\mathbf{\widehat{q}}-\boldsymbol\omega(t)^\top P \mathbf{c}\,dt\right)(x)\\
&=\boldsymbol\omega(x)^\top\eta\mathbf{e_0}+S_{2^n}\left(\int_{\xi}^{.}\boldsymbol\omega(t)^\top\,dt\right)(x)\cdot(\mathbf{\widehat{q}}-P\mathbf{c})\\
&=\boldsymbol\omega(x)^\top\eta\mathbf{e_0}+\boldsymbol\omega(x)^\top\widehat{J^\xi}^\top(\mathbf{\widehat{q}}-P\mathbf{c})\\
&=\boldsymbol\omega(x)^\top(\eta\mathbf{e_0}+\widehat{J^\xi}^\top(\mathbf{\widehat{q}}-P\mathbf{c})).  
\end{aligned}
\end{equation}
at every point of $[0,1[$.
Equation \eqref{EqMatrix2} also holds for the coefficients of Walsh polynomials. That is, we obtained the linear equation system
\begin{equation*}\label{EqLinerar1}
\mathbf{c}=\eta\mathbf{e_0}+\widehat{J^\xi}^\top(\mathbf{\widehat{q}}-P\mathbf{c})
\end{equation*}
containing the variables $c_0,c_1,\ldots,c_{2^n-1}$. In matrix form
\begin{equation}\label{EqLinear3}
({I}+\widehat{J^\xi }^\top P)\mathbf{c}
=\eta\mathbf{e_0}+\widehat{J^\xi}^\top\mathbf{\widehat{q}},
\end{equation}
where ${I}$ is the identity matrix of size $2^{n}\times 2^n$.
The unicity of solution $\overline{y}_n$ of discretized integral equation \eqref{discrete} depend on the  value  $\det({I}+\widehat{J^\xi }^\top P)$.

First, we prove the next Lemma 
\begin{lemma}\label{Lemma-Matrix-J}
	For all positive integer $n$ we have
$$	
W^{-1}\widehat{J^\xi}^{(n)\top}W=
	\begin{pmatrix}
-\frac{1}{2^{n+1}}& -\frac{1}{2^n}&-\frac{1}{2^n}&\dots&-\frac{1}{2^n}&     \frac{k^*-1}{2^n}-\xi&0&0&\dots&0 \\
0&-\frac{1}{2^{n+1}}& -\frac{1}{2^n}&\dots&-\frac{1}{2^n}&     \vdots&0&0&\dots&0 \\
\vdots&\ddots & \ddots& \ddots &\vdots &\vdots &\vdots &\vdots &\vdots &\vdots \\
0&\dots&0&-\frac{1}{2^{n+1}}&-\frac{1}{2^n}&     \frac{k^*-1}{2^n}-\xi&0&0&\dots&0 \\
0&\dots&\dots&0&-\frac{1}{2^{n+1}}&     \frac{k^*-1}{2^n}-\xi&0&0&\dots&0 \\
0&\dots&\dots&\dots&0&    \frac{2k^*-1}{2^{n+1}}-\xi &0&0&\dots &0\\	
0&\dots&\dots&\dots&0&      \frac{k^*}{2^n}-\xi&	\frac{1}{2^{n+1}}&0&\dots&0\\
0&\dots&\dots&\dots&0&      \frac{k^*}{2^n}-\xi&	\frac{1}{2^{n}}&\frac{1}{2^{n+1}}&\ddots&\vdots\\
\vdots&\dots&\dots&\dots&\vdots&  \vdots&	\vdots&\ddots&\ddots&0\\
0&0&0&\dots&0&      \frac{k^*}{2^n}-\xi&	\frac{1}{2^{n}}&\dots&\frac{1}{2^{n}}&\frac{1}{2^{n+1}}
	\end{pmatrix}.
$$
\end{lemma}
\begin{proof}
During this proof we  use the starting idea of Lemma 4 presented in paper \cite{GT}. 
We compute directly the entry $a_{ij}^\xi$ of the matrix $W^{-1}\widehat{J^\xi}^{\top}W$.

Using that $W$ is a symmetric matrix such that $W^{-1}=\frac{1}{2^n}W$ holds and applying equation \eqref{EqPaleyLemma},  we write
\begin{align*}
a_{ij}^\xi
&=\frac{1}{2^n}\sum_{k=0}^{2^n-1}\sum_{l=0}^{2^n-1}\omega_k(\frac{i}{2^n})\widehat{J^\xi}_{l,k}\omega_l(\frac{j}{2^n})\\
&=\frac{1}{2^n}\sum_{k=0}^{2^n-1}\sum_{l=0}^{2^n-1}\omega_k(\frac{i}{2^n})\int_{0}^{1}\int_{\xi}^{x}\omega_l(t)\,dt\,\omega_k(x)\,dx\,\omega_l(\frac{j}{2^n})\\
&=\frac{1}{2^n}\sum_{k=0}^{2^n-1}\sum_{l=0}^{2^n-1}\int_{0}^{1}\int_{\xi}^{x}\omega_l(t\dotplus\frac{j}{2^n})\,dt\,\omega_k(x\dotplus\frac{i}{2^n})\,dx
\\
&=\frac{1}{2^n}\int_{0}^{1}\int_{\xi}^{x}D_{2^n}(t\dotplus\frac{j}{2^n})\,dt\,D_{2^n}(x\dotplus\frac{i}{2^n})\,dx\\
&=\frac{1}{2^n}\int_{0}^{1}\int_{\xi}^{x\dotplus\frac{i}{2^n}}D_{2^n}(t\dotplus\frac{j}{2^n})\,dt\,D_{2^n}(x)\,dx\\
&=\int_{0}^{\frac{1}{2^n}}\int_{\xi}^{x\dotplus\frac{i}{2^n}}D_{2^n}(t\dotplus\frac{j}{2^n})\,dt\,dx\\
&=\int_{0}^{\frac{1}{2^n}}\int_{\xi}^{x+\frac{i}{2^n}}D_{2^n}(t\dotplus\frac{j}{2^n})\,dt\,dx.
\end{align*}
At last we used that $x\dotplus\frac{i}{2^n}=x+\frac{i}{2^n}$, while $0\leq x<\frac{1}{2^n}$. 

Now, we set $k^*$ as we did in Theorem \ref{theorem-1}.
We have three cases $k^*\leq i$, $k^*=i+1$, $k^*>i+1$, while  $\frac{i}{2^n}\leq x+\frac{i}{2^n}<\frac{i+1}{2^n}$ for $x\in [0,\frac{1}{2^n}[$.

\emph{Case I.} Let us set $k^*\leq i$.
\begin{align*}
a_{ij}^\xi
&=\int_{0}^{\frac{1}{2^n}}\int_{\xi}^{x+\frac{i}{2^n}}D_{2^n}(t\dotplus\frac{j}{2^n})\,dt\,dx\\
&=\int_{0}^{\frac{1}{2^n}}
\left(
\int_{\xi}^{\frac{k^*}{2^n}}D_{2^n}(t\dotplus\frac{j}{2^n})\,dt 
+\sum_{r=k^*}^{i-1}
\int_{\frac{r}{2^n}}^{\frac{r+1}{2^n}}D_{2^n}(t\dotplus\frac{j}{2^n})\,dt 
+
\int_{\frac{i}{2^n}}^{x+\frac{i}{2^n}}D_{2^n}(t\dotplus\frac{j}{2^n})\,dt
\right)dx\\
&=:J_1+J_2+J_3.
\end{align*}
First, we discuss the expression $J_1$.
\begin{align*}
J_1&=\int_{0}^{\frac{1}{2^n}}
\left(
\int_{\frac{k^*-1}{2^n}}^{\frac{k^*}{2^n}}D_{2^n}(t\dotplus\frac{j}{2^n})\,dt +
\int_{\frac{k^*-1}{2^n}}^{\xi}D_{2^n}(t\dotplus\frac{j}{2^n})\,dt 
\right) dx\\
&=\int_{0}^{\frac{1}{2^n}}\left(
\int_{\frac{k^*-1}{2^n}\dotplus\frac{j}{2^n}}^{\frac{k^*}{2^n}\dotplus\frac{j}{2^n}}D_{2^n}(t)\,dt+
\int_{\frac{k^*-1}{2^n}\dotplus\frac{j}{2^n}}^{\xi\dotplus\frac{j}{2^n}}D_{2^n}(t)\,dt
\right) dx\\
&=\begin{cases} 
\frac{1}{2^{n}}-(\xi - \frac{k^*-1}{2^n}),& \textrm{if } j=k^*-1,\\
0, & \textrm{otherwise}.
\end{cases}
\end{align*}

For the expression $J_2$ it is easily seen that
$$
J_2=\sum_{r=k^*}^{i-1}\int_{0}^{\frac{1}{2^n}}
\int_{\frac{r}{2^n}\dotplus \frac{j}{2^n}}^{(\frac{r}{2^n}\dotplus \frac{j}{2^n})+\frac{1}{2^n}}D_{2^n}(t)\,dt\, dx=
\begin{cases} 
\frac{1}{2^{n}},& \textrm{if } k^*\leq j<i,\\
0, & \textrm{otherwise}.
\end{cases}
$$

Analogously, it can be showed that 
$$
J_3:=\begin{cases} 
\frac{1}{2^{n+1}},& \textrm{if }i=j,\\
0, & \textrm{otherwise}.
\end{cases}
$$
Collecting our results we have that 
$$
a_{ij}^\xi=\begin{cases} 
\frac{1}{2^{n+1}},& \textrm{if }i=j,\\
\frac{1}{2^{n}},& \textrm{if } k^*\leq j<i,\\
\frac{k^*}{2^n}-\xi,& \textrm{if } j=k^*-1<i,\\
0, & \textrm{otherwise}.
\end{cases}
$$

\emph{Case II.} We set 
$k^*=i+1$. 
\begin{align*}
a_{ij}^\xi
&=\int_{0}^{\frac{1}{2^n}}\int_{\frac{i}{2^n}+(\xi-\frac{i}{2^n})}^{x+\frac{i}{2^n}}D_{2^n}(t\dotplus\frac{j}{2^n})\,dt\,dx\\
&=\int_{0}^{\frac{1}{2^n}}\int_{(\frac{i}{2^n}\dotplus\frac{j}{2^n})+(\xi-\frac{i}{2^n})}^{(\frac{i}{2^n}\dotplus\frac{j}{2^n})+x}D_{2^n}(t)\,dt\,dx
\end{align*}
(We note that $0\leq x<\frac{1}{2^n}$.) 
Equality \eqref{EqPaleyLemma} yields  $a_{ij}^\xi\neq 0$ if $i=j=k^*-1.$ 
In this case we get $\int_{(\frac{i}{2^n}\dotplus\frac{j}{2^n})+(\xi-\frac{i}{2^n})}^{(\frac{i}{2^n}\dotplus\frac{j}{2^n})+x}D_{2^n}(t)\,dt=2^n(x-(\xi -\frac{k^*-1}{2^n}))$
That is, 
$$
a_{ij}^\xi=\begin{cases} 
\frac{2k^*-1}{2^{n+1}}-\xi ,& \textrm{if } i=j=k^*-1,\\
0, & \textrm{otherwise }(i=k^*-1).
\end{cases}
$$
We note that $\frac{2k^*-1}{2^{n+1}}$ is the middle point of the interval $[\frac{k^*-1}{2^n}, \frac{k^*}{2^n}[$. 

\emph{Case III.} We set $k^*>i+1$. We have that 
$\frac{i}{2^n}\leq x+\frac{i}{2^n}<\frac{i+1}{2^n}\leq \frac{k^*-1}{2^n}\leq \xi<\frac{k^*}{2^n}$, while $0\leq x<\frac{1}{2^n}$.
 
\begin{align*}
a_{ij}^\xi
&=\int_{0}^{\frac{1}{2^n}}\int_{\xi}^{x+\frac{i}{2^n}}D_{2^n}(t\dotplus\frac{j}{2^n})\,dt\,dx\\
&=-\int_{0}^{\frac{1}{2^n}}\int_{x+\frac{i}{2^n}}^{\xi}D_{2^n}(t\dotplus\frac{j}{2^n})\,dt\,dx\\
&=-\int_{0}^{\frac{1}{2^n}}\left(
\int_{x+\frac{i}{2^n}}^{\frac{i+1}{2^n}}D_{2^n}(t\dotplus\frac{j}{2^n})\,dt
+\int_{\frac{i+1}{2^n}}^{\frac{k^*-1}{2^n}}D_{2^n}(t\dotplus\frac{j}{2^n})\,dt
+\int_{\frac{k^*-1}{2^n}}^{\xi}D_{2^n}(t\dotplus\frac{j}{2^n})\,dt
\right)dx\\
&=L_1+L_2+L_3.
\end{align*}
Now, we discuss the expression $L_1$.
$$
L_1=-\int_{0}^{\frac{1}{2^n}}
\int_{(x+\frac{i}{2^n})\dotplus\frac{j}{2^n}}^{\frac{i+1}{2^n}\dotplus\frac{j}{2^n}}D_{2^n}(t)\,dt\,dx=
-\int_{0}^{\frac{1}{2^n}}
\int_{(\frac{i}{2^n}\dotplus\frac{j}{2^n})+x}^{(\frac{i}{2^n}\dotplus\frac{j}{2^n})+\frac{1}{2^n}}D_{2^n}(t)\,dt\,dx
$$
$L_1\neq  0$ only in that case $i=j$ (see \eqref{EqPaleyLemma}) and in this case 
$
L_1=-\int_{0}^{\frac{1}{2^n}}  1-2^nx\,dx=-\frac{1}{2^{n+1}}.
$
That is, we have  
$$
L_1=\begin{cases}
-\frac{1}{2^{n+1}}, &  \textrm{if } i= j,\\
0, & \textrm{if } i\neq j.\\
\end{cases}
$$
For the expression $L_2$ we write 
$$
L_2=-\int_{0}^{\frac{1}{2^n}} 
\sum_{r=i+1}^{k^*-2}\int_{\frac{r}{2^n}\dotplus\frac{j}{2^n}}^{(\frac{r}{2^n}\dotplus\frac{j}{2^n})+\frac{1}{2^n}}D_{2^n}(t)\,dt
\,dx
=\begin{cases}
-\frac{1}{2^n},& \textrm{if } i+1\leq j\leq k^*-2,\\
0, & \textrm{otherwise.}
\end{cases}
$$
At last, we discuss the expression $L_3$.
$$
L_3= -\int_{0}^{\frac{1}{2^n}}
\int_{\frac{k^*-1}{2^n}\dotplus\frac{j}{2^n}}^{\xi\dotplus\frac{j}{2^n}}D_{2^n}(t)\,dt
dx
=\begin{cases}
\frac{k^*-1}{2^n}-\xi  , & \textrm{if } j=k^*-1,\\
0, &\textrm{otherwise. }
\end{cases}
$$
Summarizing our results we write that 
$$
a_{ij}^\xi=\begin{cases}
-\frac{1}{2^{n+1}}, &  \textrm{if } i= j,\\
-\frac{1}{2^n},& \textrm{if } i+1\leq j\leq k^*-2,\\
\frac{k^*-1}{2^n}-\xi  , & \textrm{if }i<  j=k^*-1,\\
0, &\textrm{otherwise. }
\end{cases}
$$
It completes our proof. 
\end{proof}

We mention that for $\xi=0$ we get back the result on the matrix $\widehat{J^0}^{(n)}$ proved in paper \cite{GT}.

Using this Lemma and equation \eqref{EqLinear3} we could state our next unicity theorem.
\begin{theorem}
Let $p$ and $q$ be two integrable and continuous functions defined on the interval $[0,1[$. Then
there exists a $n^*\in {\mathbb N}$, such that the  discretized integral equations \eqref{discrete} for all $n\geq n^*$  with assumption of the original initial value problem \eqref{Cauchy} has got a unique solution.
\end{theorem}
\begin{proof}
The existence of the solution of discretized integral equation follows from Theorem  \ref{theorem-1}. To prove the unicity of solution we see the matrix equation \eqref{EqLinear3} and we calculate the  value  $\det({I}+\widehat{J^\xi }^\top P)$, as we mentioned above.

Let us set $n^*$ and $k^*$ as we did in Theorem \ref{theorem-1}. 
By the diagonalization \eqref{EqDiagonalization} of the matrix $P$  and Lemma \ref{Lemma-Matrix-J} we obtain
\begin{align*}
\det(I+\widehat{J^\xi}^{\top}P)
&=\det(WIW^{-1}+WW^{-1}\widehat{J^\xi}^{\top}WD_{\tilde{p}_n}W^{-1})\\
&=\det(W(I+W^{-1}\widehat{J^\xi}^{\top}WD_{\tilde{p}_n})W^{-1})\\
&=\det(I+W^{-1}\widehat{J^\xi}^{\top}WD_{\tilde{p}_n})\\
&=\prod_{i=0}^{k^*-2}\left(  1-\frac{\tilde{p}_n(\frac{i}{2^n})}{2^{n+1}} \right)
\left( 1+ (\frac{2k^*-1}{2^{n+1}}-\xi)\tilde{p}_n(\frac{k^*-1}{2^n})\right)
\prod_{j=k^*}^{2^n-1}\left( 1+\frac{\tilde{p}_n(\frac{j}{2^n})}{2^{n+1}}\right).
\end{align*}
	
The definition  of $n^*$ gives that 
$$
\det(I+\widehat{J^\xi}^{\top}P)\neq 0 \quad \textrm{for all }n\geq n^*.
$$	
This completes the proof of this Theorem.	
\end{proof}

\section{Estimate of error}
First of all, we start with a Lemma. It discuss the behavior of a special step function, after integrating it from $\xi$ to $x$ and  applying the conditional expectation operator $S_{2^n}$ on it.
\begin{lemma} \label{lemma-step-function} Suppose that the function $f\colon [0,1[\to {\mathbb R}$ is constant on the dyadic intervals $I_n(i)$ ($i=1,\ldots , 2^n$). That is, its form is 
$$
f(x):=\sum_{k=1}^{2^n} a_k	\chi_{I_n(k)},
$$
with real numbers $a_k$ ($k=1,\ldots ,2^n$). We set $x\in I_n(i)$ for a fixed $i$ ($i=1,\ldots , 2^n$).
Then 
\begin{equation}
\label{constant-f}
S_{2^n}\left( \int_{\xi}^{.} f(t)dt\right)(x)=
\begin{cases}
a_{k^*}(\frac{k^*}{2^n}-\xi)+
\sum\limits_{k=k^*+1}^{i-1}\frac{a_k}{2^n}+\frac{a_i}{2^{n+1}}, & \textrm{if }k^*<i,\\
a_{k^*}(\frac{2k^*-1}{2^{n+1}}-\xi), & \textrm{if }k^*=i,\\
-\sum\limits_{k=i+1}^{k^*-1}\frac{a_k}{2^n}-\frac{a_i}{2^{n+1}}+a_{k^*}(\frac{k^*-1}{2^n}-\xi),  &\textrm{if }k^*>i,
\end{cases}
\end{equation}
where $k^* \in \{1,\ldots, 2^n\}$, such that $\xi \in I_n(k^*)$.
\end{lemma}
We note that $f(\frac{k-1}{2^n})=a_k$ for all $k=1,\ldots ,2^n$. 
\begin{proof}It is easily seen that 
\begin{equation*}
S_{2^n}\left( \int_{\xi}^{.} f(t)dt\right)(x)=\sum_{k=1}^{2^n}a_k S_{2^n}\left( \int_{\xi}^{.} \chi_{I_n(k)}(t)dt\right)(x).
\end{equation*}	 
We divide the sum into three parts as follows 
\begin{eqnarray*}
S_{2^n}\left( \int_{\xi}^{.} f(t)dt\right)(x)&=&\sum_{k=1}^{k^*-1}a_k S_{2^n}\left( \int_{\xi}^{.} \chi_{I_n(k)}(t)dt\right)(x)+a_{k^*} S_{2^n}\left( \int_{\xi}^{.} \chi_{I_n(k^*)}(t)dt\right)(x)\\
&&+\sum_{k=k^*+1}^{2^n}a_k S_{2^n}\left( \int_{\xi}^{.} \chi_{I_n(k)}(t)dt\right)(x).
\end{eqnarray*}	
We have three cases $k^*<i$, $k^*=i$ and $k^*>i$.

First, we discuss the case $k^*<i$. 
Using  equations \eqref{alatt}, \eqref{között} and \eqref{felett}, we immediately get 
\begin{equation}\label{step-f-1}
	S_{2^n}\left( \int_{\xi}^{.} f(t)dt\right)(x)=
0+a_{k^*}(\frac{k^*}{2^n}-\xi)+
\sum_{k=k^*+1}^{i-1}\frac{a_k}{2^n}+\frac{a_i}{2^{n+1}} .
\end{equation}

Second, we set $k^*=i$. Equations \eqref{alatt}, \eqref{között} and \eqref{felett} yield 
\begin{equation}\label{step-f-2}
	S_{2^n}\left( \int_{\xi}^{.} f(t)dt\right)(x)=
	0+a_{k^*}(\frac{2k^*-1}{2^{n+1}}-\xi)+0.
\end{equation}

At last, we set $k^*>i$.
By equations \eqref{alatt}, \eqref{között} and \eqref{felett} we write  
\begin{equation}\label{step-f-3}
S_{2^n}\left( \int_{\xi}^{.} f(t)dt\right)(x)=-\sum_{k=i+1}^{k^*-1}\frac{a_k}{2^n}-\frac{a_i}{2^{n+1}}+a_{k^*}(\frac{k^*-1}{2^n}-\xi) +0.
\end{equation}
Summarizing our results in equations \eqref{step-f-1}-\eqref{step-f-3} we get 
our Lemma.
\end{proof}

The modulus of continuity  of a function is defined by
$$
\omega_nf:=\sup\{|f(x\dotplus h)-f(x)|\colon x\in[0,1[,0\le h<2^{-n}\}.
$$
It is easily seen, that
$$|S_{ 2^n} f (x) -f (x)| \leq \omega_n f.
$$
The integral modulus of continuity is defined by 
\[
\omega^{(1)}_nf:=\sup\{\int_0^1|f(x\dotplus h)-f(x)|\,dx\colon \,0\le h<2^{-n}\}.
\]
Indeed, it is not hard to see, that
\[
\int_0^1|S_{2^n}f(x)-f(x)|\,dx\le\omega^{(1)}_nf.
\]
For more details see \cite{SchippBook}.

In this section we discuss the upper estimation of the error 
$
|y(x)-\overline{y}_n(x)|
$
for every point $x\in[0,1[$, where $y$ is the exact solution and $\overline{y}$ is the numerical solution of the Cauchy problem. As a consequence we state our main Theorem.
\begin{theorem}
Let $p$ and $q$ be two integrable and continuous functions defined on the interval $[0,1[$. Then
the solution $\overline{y}_n(x)$ of 
 the  discretized integral equation \eqref{discrete} converges uniformly to the solution of the  initial value problem \eqref{Cauchy} on the interval $[0,1[$.
\end{theorem}
\begin{proof}
In the paper \cite{GT},  $\xi=0$ was chosen. Since $\xi=0$ is a dyadic rational, so it is a left end point of a dyadic interval (it is true for every $n$). But, for such a $\xi$ which is not a dyadic rational the proof is more complicated. Moreover, the proof  has at least two parts. We have to discuss the cases while $x\in [\xi,1[$ or $x\in [0,\xi[$, it follows from the multistep algorithm.

We write 
\begin{equation}\label{EqTwoSteps}
|y(x)-\overline{y}_n(x)|\leq |y(x)-S_{2^{n}}y(x)|+|S_{2^{n}}y(x)-\overline{y}_n(x)|.
\end{equation}

First, we discuss the expression $|y-S_{2^{n}}y|$.  
It is well-known that the unique solution of Cauchy problem \eqref{Cauchy} is given by the formula
\begin{equation}\label{EqSolution}
y(x)=e^{-\int_{\xi}^{x}p(t)\,dt}\left(\eta+\int_{\xi}^{x} q(t)e^{\int_{\xi}^{t}p(s)\,ds}\,dt\right)\qquad(0\le x<1).
\end{equation}

Let us note that  the solution $y$ of the Cauchy problem \eqref{Cauchy} can be extended continuously to the close interval $[0,1]$, since the integrability of the function $p$ and $q$ ensures that the limit
\[
\lim_{x\to1^{-}}y(x)=e^{-\int_{\xi }^{1}p(t)\,dt}\left(\eta+\int_{\xi }^{1} q(t)e^{\int_{\xi }^{t}p(s)\,ds}\,dt\right)
\]
is finite. This means that the solution $y$ has finite modulus of continuity and 
\begin{equation}\label{eq-omega}
|y(x)-S_{2^n}y(x)|\le\omega_{n}y
\end{equation}
for all $x\in[0,1[$. Therefore, the first part of \eqref{EqTwoSteps} tends uniformly to zero.

Let us discuss the second part $
|S_{2^{n}}y-\overline{y}_n|
$ of inequality \eqref{EqTwoSteps}. To do this we introduce the notation 
$$
z_n(x):=\overline{y}_n(x)-S_{2^{n}}y(x) \quad \textrm{for any } x\in [0,1[.
$$
Applying equalities \eqref{integral-eq} and \eqref{discrete} we  write 
\begin{eqnarray}\label{eq-zn}
z_n(x)&=&\eta+S_{2^n}\left(\int_{\xi}^{.} S_{2^n}q(t)-S_{2^n}p(t)\overline{y}_n(t)\,dt\right)(x)-
S_{2^{n}}\left(\eta+\int_{\xi}^{.} q(t)-p(t)y(t)\,dt\right)(x)\nonumber\\
&=&S_{2^{n}}\left(\int_{\xi}^{.}S_{2^{n}}q(t)-q(t)\,dt\right)(x)-
S_{2^{n}}\left(\int_{\xi}^{.}(S_{2^{n}}p(t)-p(t))y(t)\,dt\right)(x)
\nonumber\\
&&+
S_{2^{n}}\left(\int_{\xi}^{.}S_{2^{n}}p(t)(y(t)-S_{2^{n}}y(t))\,dt\right)(x)
-S_{2^{n}}\left(\int_{\xi}^{.}S_{2^{n}}p(t)z_n(t)\,dt\right)(x)\nonumber\\
&=:&m_n^1(x)-m_n^2(x)+m_n^3(x)-S_{2^{n}}\left(\int_{\xi}^{.}S_{2^{n}}p(t)z_n(t)\,dt\right)(x)
\end{eqnarray}
for all $x\in[0,1[$. Let $m_n(x)$ be defined by
$$
m_n(x):=m_n^1(x)-m_n^2(x)+m_n^3(x).
$$
First, we estimate the expression $|m_n^2(x)|$. 
$$
|m_n^2(x)|\leq S_{2^{n}}\left|\int_{\xi}^{.}|S_{2^{n}}p(t)-p(t)||y(t)|\,dt\right|(x)
\leq \Vert y\Vert_\infty \int_0^1 |S_{2^{n}}p(t)-p(t)|\, dt
\leq \Vert y\Vert_\infty \omega_n^{(1)}p
$$
for all $x\in [0,1[$. We note that  $\Vert y\Vert_\infty$ is finite, since $y$ is a bounded function on $[0,1[$. 
Choosing $y\equiv 1$ and $p\equiv q$, we immediately get 
$$
|m_n^1(x)|\leq \omega_n^{(1)}q \quad \textrm{for all }  x\in [0,1[.
$$

Third, we estimate the expression $|m_n^3(x)|$. Set $x\in I_n(i)$.
Since,
\[
\int_{\frac{i-1}{2^{n}}}^{\frac{i}{2^{n}}}S_{2^{n}}y(t)-y(t)\,dt=0\quad(i=1,2,\ldots,2^{n}).
\]
and $S_{2^{n}}p$ is constant on  all dyadic intervals $I_n(i)$ 
for the function $m_n^3(x)$ we write 
$$
m_n^3(x)= 2^n \int_{\frac{i-1}{2^{n}}}^{\frac{i}{2^{n}}}\int_{\xi}^{x}S_{2^{n}}p(t)(y(t)-S_{2^{n}}y(t))\,dt\,dx.
$$
Set $k^*$ such that $\xi \in I_n(k^*)$ ($k^*=k^*(n)$, that is $k^*$ depends on $n$). We have three cases $k^*=i$, $k^*<i$ and $k^*>i$.
If $k^*=i$ we have 
\begin{eqnarray*}
|m_n^3(x)|&\leq& 2^n |S_{2^{n}}p(\frac{i-1}{2^n})|\int_{\frac{i-1}{2^{n}}}^{\frac{i}{2^{n}}}\left|\int_{\xi}^{x}|y(t)-S_{2^{n}}y(t)|\,dt\right|dx\\
&\leq& 2^n |S_{2^{n}}p(\frac{i-1}{2^n})\int_{\frac{i-1}{2^{n}}}^{\frac{i}{2^{n}}}\int_{\frac{i-1}{2^n}}^{\frac{i}{2^n}}|y(t)-S_{2^{n}}y(t)|\,dt\,dx\\
&\leq& \max_{0\le i<2^{n}}\left\{\frac{1}{2^n}
S_{2^{n}}p(\frac{i}{2^{n}})\right\}\omega_ny.
\end{eqnarray*}
Now, we set $k^*<i$ (that is $\xi<x$). 
\begin{eqnarray*}
&&	|m_n^3(x)|\leq \\
	&\leq & \hspace{-8pt}2^n \int_{\frac{i-1}{2^{n}}}^{\frac{i}{2^{n}}}\hspace{-5pt}
	\left(|S_{2^{n}}p(\frac{k^*-1}{2^n})|\int_{\xi}^{\frac{k^*}{2^n}}\hspace{-10pt}|y(t)-S_{2^{n}}y(t)|\,dt+|S_{2^{n}}p(\frac{i-1}{2^n})|\int_{\frac{i-1}{2^n}}^{x}\hspace{-10pt}|y(t)-S_{2^{n}}y(t)|\,dt\right)dx\\
	&\leq&\hspace{-8pt} 2^n \int_{\frac{i-1}{2^{n}}}^{\frac{i}{2^{n}}}\hspace{-5pt}
	\left(|S_{2^{n}}p(\frac{k^*-1}{2^n})|\int_{\frac{k^*-1}{2^n}}^{\frac{k^*}{2^n}}\hspace{-10pt}|y(t)-S_{2^{n}}y(t)|\,dt+|S_{2^{n}}p(\frac{i-1}{2^n})|\int_{\frac{i-1}{2^n}}^{\frac{i}{2^n}}\hspace{-10pt}|y(t)-S_{2^{n}}y(t)|\,dt\right)dx\\
	&\leq&\hspace{-8pt} 2 \max_{0\le i<2^{n}}\left\{\frac{1}{2^n}
	S_{2^{n}}p(\frac{i}{2^{n}})\right\}\omega_ny.
\end{eqnarray*}
Analogously, for $k^*>i$ (that is $\xi>x$) we get 
$$
|m_n^3(x)|\leq 2\max_{0\le i<2^{n}}\left\{\frac{1}{2^n}
S_{2^{n}}p(\frac{i}{2^{n}})\right\}\omega_ny.
$$
Summarizing our results on $|m_n^i(x)|$ ($i=1,2,3$) we have that
\begin{equation}\label{eq-Mn}
m_n(x)\leq \omega^{(1)}_{n}q+\|y\|_{\infty}\omega^{(1)}_{n}p+2\max_{0\le i<2^{n}}\left\{\frac{1}{2^n}
S_{2^{n}}p(\frac{i}{2^{n}})\right\}\omega_ny=:M_n
\end{equation}
for all $x\in[0,1[$. 
By \eqref{main-eq} the sequence $M_n$ tends to zero if $n\to\infty$.

Since, 
\begin{equation}\label{eq-zn-2}
z_n(x)=m_n(x)-S_{2^{n}}\left(\int_{\xi}^{.}S_{2^{n}}p(t)z_n(t)\,dt\right)(x),
\end{equation}
we have to discuss the last expression on the right side of the equation.
The functions $S_{2^{n}}p$, $m_n$ and $z_n$ are constants on the dyadic intervals $I_n(i)$ for all $i=1,2,\ldots,2^{n}$. Hence, we apply Lemma \ref{lemma-step-function} for cases $k^*=i$, $k^*<i$ and $k^*>i$.

First, we discuss cases $k^*=i$ and  $k^*<i$. 

\emph{Case $k^*=i$} ($x\in I_n(i)$).
\begin{equation*}\label{i-egyenlő-1}
S_{2^{n}}\left(\int_{\xi}^{.}S_{2^{n}}p(t)z_n(t)\,dt\right)(x)=S_{2^{n}}p(\frac{k^*-1}{2^n})z_n(\frac{k^*-1}{2^n})\left( \frac{2{k^*}-1}{2^{n+1}}-\xi\right).
\end{equation*}
From this we immediately get 
\begin{equation*}
z_n(\frac{k^*-1}{2^n})=m_n(\frac{k^*-1}{2^n})-S_{2^{n}}p(\frac{k^*-1}{2^n})z_n(\frac{k^*-1}{2^n})\left( \frac{2{k^*}-1}{2^{n+1}}-\xi\right)
\end{equation*}
and 
\begin{equation}\label{i-egyenlő-2}
z_n(\frac{k^*-1}{2^n})=\frac{m_n(\frac{k^*-1}{2^n})}{1+S_{2^{n}}p(\frac{k^*-1}{2^n})\left( \frac{2{k^*}-1}{2^{n+1}}-\xi\right)}.
\end{equation}
We note that the denominator is not 0, if $n$ is big enough (see later).

\emph{Case $k^*<i$} ($x\in I_n(i)$).
\begin{eqnarray*}\label{i-kisebb-1}
&&S_{2^{n}}\left(\int_{\xi}^{.}S_{2^{n}}p(t)z_n(t)\,dt\right)(x)=S_{2^{n}}p(\frac{k^*-1}{2^n})z_n(\frac{k^*-1}{2^n})\left( \frac{{k^*}}{2^{n}}-\xi\right)\nonumber\\
&&\hspace{1cm}+\frac{1}{2^{n}}\sum_{k=k^*+1}^{i-1}S_{2^{n}}p(\frac{k-1}{2^{n}})z_n(\frac{k-1}{2^{n}})+\frac{1}{2^{n+1}}S_{2^{n}}p(\frac{i-1}{2^{n}})z_n(\frac{i-1}{2^{n}}).
\end{eqnarray*}
This yields
\begin{eqnarray*}
z_n(\frac{i-1}{2^n})&=&m_n(\frac{i-1}{2^n})-\frac{1}{2^{n}}\sum_{k=k^*+1}^{i-1}S_{2^{n}}p(\frac{k-1}{2^{n}})z_n(\frac{k-1}{2^{n}})\nonumber\\
&&-\frac{1}{2^{n+1}}S_{2^{n}}p(\frac{i-1}{2^{n}})z_n(\frac{i-1}{2^{n}})
-S_{2^{n}}p(\frac{k^*-1}{2^n})z_n(\frac{k^*-1}{2^n})\left( \frac{{k^*}}{2^{n}}-\xi\right)
\end{eqnarray*}
and
\begin{eqnarray}\label{i-kisebb-2}
z_n(\frac{i-1}{2^n})\left(1+\frac{1}{2^{n+1}}S_{2^{n}}p(\frac{i-1}{2^{n}})\right)&=&m_n(\frac{i-1}{2^n})-\frac{1}{2^{n}}\sum_{k=k^*+1}^{i-1}S_{2^{n}}p(\frac{k-1}{2^{n}})z_n(\frac{k-1}{2^{n}})\nonumber\\
&&-S_{2^{n}}p(\frac{k^*-1}{2^n})z_n(\frac{k^*-1}{2^n})\left( \frac{{k^*}}{2^{n}}-\xi\right)
\end{eqnarray}
Applying equations \eqref{i-egyenlő-2} and \eqref{i-kisebb-2} and mathematical induction we have
\begin{eqnarray}\label{rekurziv-formula}
z_n(\frac{i-1}{2^n})\left(1+\frac{1}{2^{n+1}}S_{2^{n}}p(\frac{i-1}{2^{n}})\right)&=&m_n(\frac{i-1}{2^n})
-\sum_{k=k^*+1}^{i-1} \rho^{(n)}_{k-1}m_n(\frac{k-1}{2^{n}})
\prod_{j=k+1}^{i-1}\left(1-\rho^{(n)}_{j-1}\right)\nonumber\\
&&-m_n(\frac{k^*-1}{2^{n}})\sigma^{(n)}_\xi
\prod_{j=k^*+1}^{i-1}\left(1-\rho^{(n)}_{j-1}\right),
\end{eqnarray}
where 
$$
\rho^{(n)}_k:=\frac{\frac{1}{2^n}S_{2^{n}}p(\frac{k}{2^{n}})}{1+\frac{1}{2^{n+1}}S_{2^{n}}p(\frac{k}{2^{n}})},
\quad 
\sigma^{(n)}_\xi:=\frac{S_{2^{n}}p(\frac{k^*-1}{2^{n}})\left(\frac{k^*}{2^n}-\xi\right)}{1+S_{2^{n}}p(\frac{k^*-1}{2^{n}})\left(\frac{2k^*-1}{2^{n+1}}-\xi\right)}
$$
for all $k>k^*$  (and $i>k^*$). This yields 
\begin{equation}
\begin{split}
|z_n(\frac{i-1}{2^n})|&\left|1+\frac{1}{2^{n+1}}S_{2^{n}}p(\frac{i-1}{2^{n}})\right|\leq M_n
+\sum_{k=k^*+1}^{i-1} |\rho^{(n)}_{k-1}|M_n
\prod_{j=k+1}^{i-1}\left(1+|\rho^{(n)}_{j-1}|\right)\\
&+M_n|\sigma^{(n)}_\xi|
\prod_{j=k^*+1}^{i-1}\left(1+|\rho^{(n)}_{j-1}|\right)\\
=& M_n\prod_{j=k^*+1}^{i-1}\left(1+|\rho^{(n)}_{j-1}|\right)
+M_n|\sigma^{(n)}_\xi|
\prod_{j=k^*+1}^{i-1}\left(1+|\rho^{(n)}_{j-1}|\right)\\
\leq & M_n (1+|\sigma^{(n)}_\xi|)\prod_{j=1}^{i-1}\left(1+|\rho^{(n)}_{j-1}|\right)).
\end{split}
\end{equation}
First, we estimate the expression $|\sigma^{(n)}_\xi|$. By inequalities 
\eqref{main-eq} and \eqref{main-eq-2}, we could choose a natural number $n^*$, such that 
$$
\left|S_{2^{n}}p(\frac{k^*-1}{2^{n}})\left(\frac{k^*}{2^n}-\xi\right)\right|
\leq \frac{1}{2^n}\left|S_{2^{n}}p(\frac{k^*-1}{2^{n}})\right|<\frac{1}{2}
$$
for all $n\geq n^*$.
\begin{equation}\label{eq-1-4}
\left|S_{2^{n}}p(\frac{k^*-1}{2^{n}})\left(\frac{2k^*-1}{2^{n+1}}-\xi\right)\right|\leq 
|\frac{1}{2^{n+1}}S_{2^{n}}p(\frac{k^*-1}{2^{n}})|<\frac{1}{4}
\end{equation}
for all $n\geq n^*$. Analogically to inequality \eqref{main-eq-3}, we get 
$$
|\sigma^{(n)}_\xi|\leq \frac{2}{3}\quad \textrm{for } n\geq n^*
$$
and 
$$
	|z_n(\frac{i-1}{2^n})|\left|1+\frac{1}{2^{n+1}}S_{2^{n}}p(\frac{i-1}{2^{n}})\right|\leq
	2M_n \prod_{j=1}^{i-1}\left(1+|\rho^{(n)}_{j-1}|\right)).
$$
Applying  inequalities \eqref{main-eq} and  \eqref{main-eq-2} for $n\geq n^*$
in 	paper \cite{GT} it is proved 
\begin{equation}
\prod_{j=1}^{i-1}\left(1+|\rho^{(n)}_{j-1}|\right))\leq 
e^{2\int_{0}^{1}|p(x)|\,dx}.
\end{equation}
Using this we get 
$$\frac{1}{2}|z_n(x)|
=\frac{1}{2}|z_n(\frac{i-1}{2^{n}})|< 2M_n e^{2\int_{0}^{1}|p(x)|\,dx}
$$
for all $x\in I_n(i)$ and $n\geq n^*$ ($i>k^*$).
Since, the right side is independent from $x$, we write
$$
|z_n(x)|< 4M_n e^{2\int_{0}^{1}|p(x)|\,dx}
$$ 
for all $x\in [\frac{k^*}{2^n},1[$ and $n\geq n^*$.

We discuss the case $i=k^*$. From inequalities \eqref{i-egyenlő-2} \eqref{eq-1-4} we have 
$$
z_n(x)=z_n(\frac{k^*-1}{2^n})< \frac{4}{3}M_n
$$
for all $x\in I_n(k^*)$ and $n\geq n^*$.

Summarizing our results
\begin{equation}\label{eq-half-1}
|z_n(x)|< M_n \left( 4e^{2\int_{0}^{1}|p(x)|\,dx}+\frac{4}{3}\right)
\end{equation}
for all $x\in [\xi ,1[$  (more exactly $x\in [\frac{k^*-1}{2^n},1[$) and $n\geq n^*$.

\emph{Case $k^*>i$} ($x\in I_n(i)$). 
The functions $S_{2^{n}}p$, $m_n$ and $z_n$ are constants on the dyadic intervals $I_n(i)$ for all $i=1,2,\ldots,2^{n}$. Hence, we apply Lemma \ref{lemma-step-function} for equation \eqref{eq-zn-2}
\begin{eqnarray*}\label{i-nagyobb-1}
	&&S_{2^{n}}\left(\int_{\xi}^{.}S_{2^{n}}p(t)z_n(t)\,dt\right)(x)=
	S_{2^{n}}p(\frac{k^*-1}{2^n})z_n(\frac{k^*-1}{2^n})\left( \frac{{k^*-1}}{2^{n}}-\xi\right)\nonumber\\
	&&\hspace{1cm}-\frac{1}{2^{n}}\sum_{k=i+1}^{k^*-1}S_{2^{n}}p(\frac{k-1}{2^{n}})z_n(\frac{k-1}{2^{n}})-\frac{1}{2^{n+1}}S_{2^{n}}p(\frac{i-1}{2^{n}})z_n(\frac{i-1}{2^{n}}).
\end{eqnarray*}
For equality \eqref{eq-zn-2} we get
\begin{eqnarray*}
z_n(\frac{i-1}{2^n})	&=&m_n(\frac{i-1}{2^n})-
	S_{2^{n}}p(\frac{k^*-1}{2^n})z_n(\frac{k^*-1}{2^n})\left( \frac{{k^*-1}}{2^{n}}-\xi\right)\nonumber\\
	&&+\frac{1}{2^{n}}\sum_{k=i+1}^{k^*-1}S_{2^{n}}p(\frac{k-1}{2^{n}})z_n(\frac{k-1}{2^{n}})+\frac{1}{2^{n+1}}S_{2^{n}}p(\frac{i-1}{2^{n}})z_n(\frac{i-1}{2^{n}}).
\end{eqnarray*}	
That is, 
\begin{eqnarray}\label{i-nagyobb-2}
	z_n(\frac{i-1}{2^n})\left( 1-\frac{1}{2^{n+1}}S_{2^n}p(\frac{i-1}{2^n})\right)	&=&m_n(\frac{i-1}{2^n})-
	S_{2^{n}}p(\frac{k^*-1}{2^n})z_n(\frac{k^*-1}{2^n})\left( \frac{{k^*-1}}{2^{n}}-\xi\right)\nonumber\\
	&&+\frac{1}{2^{n}}\sum_{k=i+1}^{k^*-1}S_{2^{n}}p(\frac{k-1}{2^{n}})z_n(\frac{k-1}{2^{n}}).
\end{eqnarray}	
Using \eqref{i-egyenlő-2} we get 
\begin{eqnarray*}
z_n(\frac{i-1}{2^n})\left( 1-\frac{1}{2^{n+1}}S_{2^n}p(\frac{i-1}{2^n})\right)	&=&m_n(\frac{i-1}{2^n})+\frac{1}{2^{n}}\sum_{k=i+1}^{k^*-1}S_{2^{n}}p(\frac{k-1}{2^{n}})z_n(\frac{k-1}{2^{n}})\nonumber\\
&&-m_n(\frac{k^*-1}{2^n})\delta_\xi^{(n)}, 
\end{eqnarray*}
with 
$$
\delta_\xi^{(n)}:=\frac{S_{2^{n}}p(\frac{k^*-1}{2^n})\left( \frac{{k^*-1}}{2^{n}}-\xi\right)}{1+S_{2^{n}}p(\frac{k^*-1}{2^n})\left( \frac{{2k^*-1}}{2^{n}}-\xi\right)}.
$$
By mathematical induction we get 
\begin{eqnarray}\label{i-nagyobb-3}
	z_n(\frac{i-1}{2^n})\left( 1-\frac{1}{2^{n+1}}S_{2^n}p(\frac{i-1}{2^n})\right)	&=&m_n(\frac{i-1}{2^n})+\sum_{k=i+1}^{k^*-1}m_n(\frac{k-1}{2^{n}})\rho_{k-1}^{(n)}\prod_{j=i+1}^{k-1}(1+\rho_{j-1}^{(n)})\nonumber\\
	&&-m_n(\frac{k^*-1}{2^n})\delta_\xi^{(n)}\prod_{j=i+1}^{k^*-1}(1+\rho_{j-1}^{(n)}), 
\end{eqnarray}
where
$$
\rho^{(n)}_k:=\frac{\frac{1}{2^n}S_{2^{n}}p(\frac{k}{2^{n}})}{1-\frac{1}{2^{n+1}}S_{2^{n}}p(\frac{k}{2^{n}})} \quad \textrm{for } k<k^*.
$$
Equality \eqref{i-nagyobb-3} yields 
\begin{eqnarray*}
|z_n(\frac{i-1}{2^n})|\left| 1-\frac{1}{2^{n+1}}S_{2^n}p(\frac{i-1}{2^n})\right|	&\leq&M_n+\sum_{k=i+1}^{k^*-1}M_n|\rho_{k-1}^{(n)}|\prod_{j=i+1}^{k-1}(1+|\rho_{j-1}^{(n)}|)\nonumber\\
&&+M_n|\delta_\xi^{(n)}|\prod_{j=i+1}^{k^*-1}(1+|\rho_{j-1}^{(n)}|)\nonumber\\
&=&M_n(1+|\delta_\xi^{(n)}|)\prod_{j=i+1}^{k^*-1}(1+|\rho_{j-1}^{(n)}|).
\end{eqnarray*}
Analogically to inequality \eqref{main-eq-3}, we get 
$$
|\delta^{(n)}_\xi|\leq \frac{2}{3}\quad \textrm{for } n\geq n^*
$$
and 
\begin{eqnarray*}
\frac{1}{2}|z_n(\frac{i-1}{2^n})|<|z_n(\frac{i-1}{2^n})|\left| 1-\frac{1}{2^{n+1}}S_{2^n}p(\frac{i-1}{2^n})\right|	&\leq&2M_n\prod_{j=i+1}^{k^*-1}(1+|\rho_{j-1}^{(n)}|)
\end{eqnarray*}
for $n\geq n^*$.  

For $n\geq n^*$ and $k<k^*$ we have $|\rho_k^{(n)}|\leq \frac{1}{2^{n-1}}|S_{2^n}p(\frac{k}{2^n})|$ and 
\begin{eqnarray*}
\prod_{j=i+1}^{k^*-1}(1+|\rho_{j-1}^{(n)}|)
&\leq& \prod_{j=i+1}^{k^*-1}(1+\frac{1}{2^{n-1}}|S_{2^n}p(\frac{j-1}{2^n})|)\\
&=& \prod_{j=i+1}^{k^*-1}\left(1+2\left|\int_{\frac{j-1}{2^n}}^{\frac{j}{2^n}}p(x)dx\right|\right)\\
&\leq &\prod_{j=i}^{k^*-1}\left(1+2\int_{\frac{j-1}{2^n}}^{\frac{j}{2^n}}\left|p(x)\right|dx\right).
\end{eqnarray*}
The inequality between the arithmetic and geometric means yields 
\begin{eqnarray*}
	\prod_{j=i+1}^{k^*-1}(1+|\rho_{j-1}^{(n)}|)
	&\leq &\left(\frac{1}{k^*-i} \sum_{j=i}^{k^*-1}\left(1+2\int_{\frac{j-1}{2^n}}^{\frac{j}{2^n}}\left|p(x)\right|dx\right)\right)^{k^*-i}\\
	&\leq& \left(1+\frac{2}{k^*-i}\int_{0}^{1}\left|p(x)\right|dx \right)^{k^*-i}\\
	&<&e^{2\int_{0}^{1}\left|p(x)\right|dx}.
\end{eqnarray*}
That is, we get
\begin{equation}\label{eq-half-2}
|z_n(x)|=|z_n(\frac{i-1}{2^n})|\leq 4M_ne^{2\int_{0}^{1}\left|p(x)\right|dx}
\end{equation}
for all $x\in I_n(i)$ ($k^*>i$) and $n\geq n^*$. 
Summarizing our results in inequalities \eqref{eq-half-1} and \eqref{eq-half-2} we have 
\begin{equation}\label{eq-total}
|z_n(x)|< M_n \left( 4e^{2\int_{0}^{1}|p(x)|\,dx}+\frac{4}{3}\right)
\end{equation}
for all $x\in [0 ,1[$ and $n\geq n^*$.

Collecting the results of inequalities \eqref{EqTwoSteps}, \eqref{eq-omega}, \eqref{eq-zn}, \eqref{eq-Mn}, \eqref{eq-zn-2} and \eqref{eq-total}, while $n\to \infty$ we get the statement of our main Theorem. 
\end{proof}

\section{Examples for numerical solution of  Cauchy initial value problem}

In our first example  $p(x)$ is constant and $q(x)$ is bounded. Namely, we discuss the initial value problem 
\begin{equation}\label{example-1}
\begin{aligned}y'+y&=(x+1)^2,\\
y(1/2)&=5/4.
\end{aligned}
\end{equation}
The exact solution of initial value problem \eqref{example-1} is $y(x)=x^2+1$. The application of multistep algorithm is showed in Figure \ref{figure-1}. 
\begin{figure}[h]
	\centering
	\includegraphics[width=10cm]{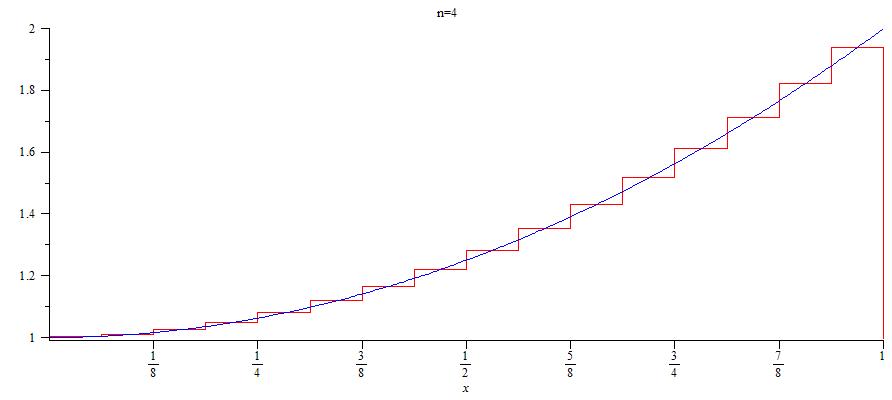}
	\caption{The numerical solution $\overline{y}_4$ of the Cauchy problem \eqref{example-1} for $n=4$.}\label{figure-1}
\end{figure}
The calculations were exact and fast. The algorithm works properly, using our theorems, the numerical solution $\overline{y}_n$ converges uniformly to the exact solution of the Cauchy problem. The supremum of the absolute difference between the numerical solution  $\overline{y}_n$ and the exact solution $y$ is reduced almost by half if the value of $n$ increased by one, as you can see in Table \ref{table-1}.  
{\tiny
	\begin{table}[h]
		\begin{tabular}{|c||c|c|c|c|c|c|c|c|}
			\hline
			n & $0\leq x<\frac{1}{8}$ & $\frac{1}{8}\leq x<\frac{2}{8}$ & $\frac{2}{8}\leq x<\frac{3}{8}$ & $\frac{3}{8}\leq x<\frac{4}{8}$ & $\frac{4}{8}\leq x<\frac{5}{8}$ & $\frac{5}{8}\leq x<\frac{6}{8}$ & $\frac{6}{8}\leq x<\frac{7}{8}$ & $\frac{7}{8}\leq x<1$\\
			\hline
			5& 0.00349579& 0.00737377& 0.01125507& 0.01513930& 0.01905982& 0.02298321& 0.02690459& 0.03082420			
			\\
			\hline
			6&0.00185003 & 0.00379615& 0.00574309& 0.00769075& 0.00964805& 0.01160543& 0.01356231& 0.01551875\\
			\hline
			7&0.00095073& 0.00192555& 0.00290057& 0.00387577& 0.00485346& 0.00583108& 0.00680857& 0.00778590\\
			\hline
			8& 0.00048182& 0.00096966& 0.00145756& 0.00194550& 0.00243407& 0.00292262& 0.00341113& 0.00389962\\
			\hline
			9& 0.00024252& 0.00048656& 0.00073060& 0.00097466& 0.00121887& 0.00146308& 0.00170728& 0.00195147
			\\
			\hline
			10&0.00012167& 0.00024371& 0.00036576& 0.00048780& 0.00060989& 0.00073198& 0.00085407& 0.00097615
			\\
			\hline
		\end{tabular}
		\caption{Estimate of $\sup |\overline{y}_n(x)-y(x)|$ on the dyadic intervals of length 1/8 for Cauchy problem \eqref{example-1}. }\label{table-1}
	\end{table}
}

In our second example we deal with the numerical solution of the initial value problem \eqref{example}. 
In this case, only the multistep algorithm works (see Figure \ref{figure-3}), because the integrability of the functions $p(x)$  and $q(x)$ is essential to calculate their Fourier coefficients which appear in the linear system. 
\begin{figure}[h]
	\centering
	\includegraphics[width=10cm]{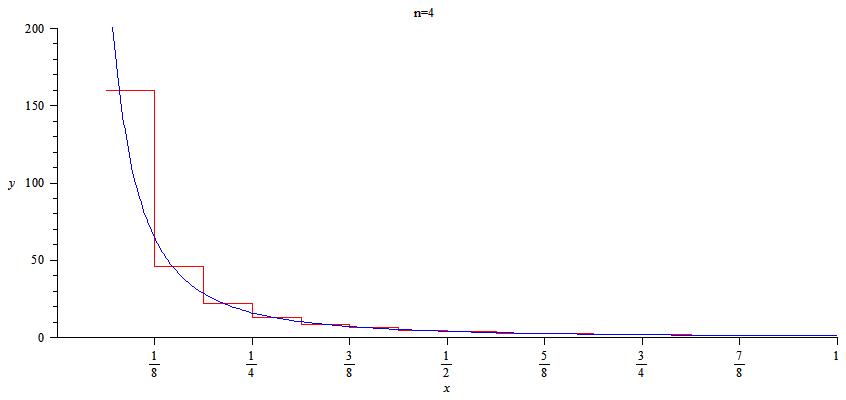}
	\caption{The numerical solution $\overline{y}_4$ of the Cauchy problem \eqref{example} for $n=4$.}\label{figure-3}
\end{figure}
Function $q(x)$ is not integrable on the interval $[0,1[$, but it is integrable on every interval of the form $[b,1[$ with $b>0$. 
The numerical solution $\overline{y}_n(x)$ is undefined on the first interval of the form $[0,1/2^n[$, because the function $q(x)$ is not integrable on it. Since the length of this interval is $1/2^n\to 0$, thus the domain $[1/2^n,1[$ of the numerical solution is approaching to the interval $[0,1[$, while the value of $\overline{y}_n$ at every point converge to the value of  exact solution $y$.
The supremum of the absolute difference between the numerical solution $\overline{y}_n$ and the exact solution  $y$ for some values of $n$ is showed in Table \ref{table-3}.
{\tiny
	\begin{table}[h]
		\begin{tabular}{|c||c|c|c|c|c|c|c|c|}
			\hline
			n & $0\leq x<\frac{1}{8}$ & $\frac{1}{8}\leq x<\frac{2}{8}$ & $\frac{2}{8}\leq x<\frac{3}{8}$ & $\frac{3}{8}\leq x<\frac{4}{8}$ & $\frac{4}{8}\leq x<\frac{5}{8}$ & $\frac{5}{8}\leq x<\frac{6}{8}$ & $\frac{6}{8}\leq x<\frac{7}{8}$ & $\frac{7}{8}\leq x<1$\\
			\hline
5&undefined& 11.52427962& 1.68047727& 0.52643361& 0.22842141& 0.11934953 &0.07021601& 0.04488844
\\
\hline
6& undefined& 6.71718816& 0.91388203& 0.27888157& 0.11938173& 0.06176085& 0.03604064& 0.02286893
\\
\hline
7&undefined& 3.65427433& 0.47760200& 0.14367669& 0.06106587& 0.03143021& 0.01826528& 0.01154620
\\
\hline
8& undefined& 1.91008439& 0.24428566& 0.07294092& 0.03088767& 0.01585627& 0.00919542& 0.00580176
\\
\hline
9&undefined& 0.97706041& 0.12355664& 0.03675180& 0.01553393& 0.00796391& 0.00461360& 0.00290813
\\
\hline
10& undefined& 0.49420584& 0.06213729& 0.01844696& 0.00778967& 0.00399096& 0.00231080& 0.00145589
\\
\hline
\end{tabular}
\caption{Estimate of $\sup |\overline{y}_n(x)-y(x)|$ on the dyadic intervals of length 1/8 for Cauchy problem \eqref{example}. }\label{table-3}
\end{table}
}			
			
Although we wrote in the first column of Table \ref{table-3} that "undefined" value on the interval $[0,1/8[$, but $\sup |\overline{y}_n(x)-y(x)|$ can be calculated  on some subintervals of $[0,1/8[$. For example in case $n=4$ (see Figure \ref{figure-3}), $\sup |\overline{y}_n(x)-y(x)|$ is undefined only at the subinterval 	$[0,\frac{1}{16}[ $ and it can be determined at the subinterval $[\frac{1}{16},\frac{2}{16}[ $. For a big $n$ the expression $\sup |\overline{y}_n(x)-y(x)|$ is undefined only at a subinterval $[0,\frac{1}{2^n}[$ and outside it is  finite.

\section{Acknowledgement}	
The author thanks support of project  GINOP-2.2.1-15-2017-00055 and  possibility of applying  the worksheets improved for it. Moreover, the author thanks Toledo for valuable advises.

\end{document}